\documentclass[12pt]{amsart}

\usepackage[pdfauthor   = {Mohamed\ Barakat\ and\ Markus\ Lange-Hegermann},
            pdftitle    = {Characterizing\ Serre\ quotients\ with\ no\ section\ functor\ and\ applications\ to\ coherent\ sheaves},
            pdfsubject  = {},
            pdfkeywords = {Constructive\ homological\ algebra; Serre\ quotient; Gabriel\ localization; coherent\ sheaves;},
            bookmarks=true,
            bookmarksopen=true,
            pagebackref=true,
            hyperindex=true,
            colorlinks=true,
            linkcolor=blue,
            citecolor=blue,
            filecolor=blue,
            urlcolor=blue,
            ]{hyperref}

\usepackage[utf8,utf8x]{inputenc}
\usepackage[english]{babel}
\usepackage[T1]{fontenc}
\usepackage{geometry}                
\usepackage{times}
\usepackage[novbox]{pdfsync}
\usepackage{mathrsfs}
\usepackage{latexsym}
\usepackage{amssymb}
\usepackage{amsthm}
\usepackage{epsfig}
\usepackage{colortbl}
\usepackage[all]{xy}
\usepackage{fancyvrb}
\usepackage{graphicx}
\usepackage[dvipsnames]{xcolor}
\usepackage{accents} 
\usepackage{enumerate}

\renewcommand\theenumi{\alph{enumi}}
\renewcommand{\labelenumi}{(\theenumi)}

\usepackage{wrapfig}
\usepackage{tikz}
\usetikzlibrary{shapes,arrows,matrix,backgrounds,positioning,plotmarks,calc,patterns,matrix,
decorations.shapes,
decorations.fractals,
decorations.markings,
decorations.pathreplacing,
decorations.pathmorphing,
decorations.text
}
\usepackage[colorinlistoftodos,shadow]{todonotes}

\usepackage{multirow}
\usepackage{mdwlist}
\usepackage{stmaryrd}
\usepackage{mathdots} 

\usepackage[toc,page]{appendix}
\usepackage{float}

\newtheoremstyle{mytheoremstyle} 
    {5pt}                    
    {5pt}                    
    {\itshape}                   
    {\parindent}                           
    {\bf}                   
    {.}                          
    {.5em}                       
    {}  

\theoremstyle{mytheoremstyle}

\newtheorem{theorem}{Theorem}[section]

\newtheorem{lemm}[theorem]{Lemma}
\newtheorem{prop}[theorem]{Proposition}
\newtheorem{coro}[theorem]{Corollary}

\newtheoremstyle{mytdefintionstyle} 
    {5pt}                    
    {5pt}                    
    {\rm}                   
    {\parindent}                           
    {\bf}                   
    {.}                          
    {.5em}                       
    {}  

\theoremstyle{remark}
\newtheorem{rmrk}[theorem]{Remark}

\theoremstyle{mytdefintionstyle}

\newtheorem{exmp}[theorem]{Example}

\newtheoremstyle{exmp_contd}
    {5pt}                    
    {5pt}                    
    {\rm}                   
    {\parindent}                           
    {\bf}                   
    {.}                          
    {.5em}                       
    {\thmname{#1}\ \thmnumber{ #2}\thmnote{#3}\ (continued)}  
\theoremstyle{exmp_contd}

\newcommand\nameft\textrm

\newcommand{\QQ}{{\mathscr{Q}}}
\renewcommand{\SS}{{\mathscr{S}}}
\newcommand{\FF}{{\mathscr{F}}}
\newcommand{\GG}{{\mathscr{G}}}
\newcommand{\HH}{{\mathscr{H}}}

\DeclareMathOperator{\coker}{coker}

\DeclareMathOperator{\img}{im}
\DeclareMathOperator{\conimg}{conim}

\DeclareMathOperator{\Hom}{Hom}

\DeclareMathOperator{\Proj}{Proj}
\DeclareMathOperator{\Sh}{Sh}

\DeclareMathOperator{\Obj}{Obj}

\DeclareMathOperator{\Sat}{Sat}

\newcommand{\Sgrmod}{{\grS\mathrm{\textnormal{-}grMod}}}

\newcommand{\Sfpgrmod}{{\grS\mathrm{\textnormal{-}grmod}}}

\newcommand{\Sfpgrmodd}{{\grS\mathrm{\textnormal{-}grmod}_{\geq d}}}
\newcommand{\Sqfgrmod}{{\grS\mathrm{\textnormal{-}qfgrmod}}}

\newcommand{\qCoh}{\mathfrak{qCoh}\,}
\newcommand{\Coh}{\mathfrak{Coh}\,}

\newcommand\grS{S}

\newcommand\grM{M}

\newcommand\shF{\mathcal{F}}

\renewcommand\O{\mathcal{O}}
\newcommand\PP{\mathbb{P}}

\newcommand\A{\mathcal{A}}
\newcommand\C{\mathcal{C}}
\newcommand{\D}{\mathcal{D}}

\newcommand\B{\mathcal{B}}

\newcommand\F{\mathcal{F}}

\renewcommand{\O}{\mathcal{O}}

\newcommand{\CC}{\mathbb{C}}

\newcommand{\Z}{\mathbb{Z}}

\renewcommand\phi{\varphi}

\DeclareMathOperator\Id{Id}

\DeclareMathOperator\Cl{Cl}

\definecolor{darkgray}{rgb}{0.3,0.3,0.3}

\newcommand{\biggerA}{{\A'}}
\newcommand{\biggerC}{{\C'}}
\newcommand{\bigA}{{\widetilde{\A}}}
\newcommand{\bigC}{{\widetilde{\C}}}
\newcommand{\ourA}{{\A}}
\newcommand{\ourC}{{\C}}

\pgfarrowsdeclarecombine[\pgflinewidth]
  {doublestealth}{doublestealth}{stealth'}{stealth'}{stealth'}{stealth'}

\definecolor{darkgreen}{rgb}{0.008,0.617,0.067}
\definecolor{brown}{rgb}{0.6,0.4,0.2}

\newif\ifjournalversion

\author{Mohamed Barakat}
\address{Department of mathematics, University of Kaiserslautern, 67653 Kaiserslautern, Germany}
\email{\href{mailto:Mohamed Barakat <barakat@mathematik.uni-kl.de>}{barakat@mathematik.uni-kl.de}}

\author{Markus Lange-Hegermann}
\address{Lehrstuhl B f\"ur Mathematik, RWTH Aachen University, 52062 Aachen, Germany}
\email{\href{mailto:Markus Lange-Hegermann <markus.lange.hegermann@rwth-aachen.de>}{markus.lange.hegermann@rwth-aachen.de}}

\begin{document}

\title[Characterizing \nameft{Serre} quotients with no section functor]{Characterizing \nameft{Serre} quotients with no section functor and applications to coherent sheaves}

\begin{abstract}
  We prove an analogon of the the fundamental homomorphism theorem for certain classes of exact and essentially surjective functors of \nameft{Abel}ian categories $\QQ:\ourA \to \B$.
  It states that $\QQ$ is up to equivalence the \nameft{Serre} quotient $\ourA \to \ourA / \ker \QQ$, even in cases when the latter does not admit a section functor.
  For several classes of schemes $X$, including projective and  toric varieties, this characterization applies to the sheafification functor from a certain category $\A$ of finitely presented graded modules to the category $\B=\Coh X$ of coherent sheaves on $X$.
  This gives a direct proof that $\Coh X$ is a \nameft{Serre} quotient of $\A$.
\end{abstract}


\maketitle


\renewcommand\theenumi{\alph{enumi}}
\renewcommand{\labelenumi}{(\theenumi)}

 \section{Introduction}

An essentially surjective exact functor $\QQ:\A \to \B$ of \nameft{Abel}ian categories induces an essentially surjective exact embedding $\overline{\QQ}:\A/ \C \to \B$, where $\C$ is the kernel of $\QQ$ and $\A/\C$ the \nameft{Serre} quotient.
It is natural to ask if or when $\overline{\QQ}$ is an equivalence of categories, i.e., if or when the fundamental homomorphism Theorem is valid for exact functors of \nameft{Abel}ian categories.

\nameft{Grothendieck} mentioned this question in \cite{Tohoku_en} but did not provide an answer:
\begin{quote}
Thus $\A/\C$ \emph{appears as an abelian category; moreover the identity functor $\QQ: \A \to \A/\C$ is exact} (and, in particular, commutes with kernels, cokernels, images, and coimages), \emph{$\QQ(A) = 0$ if and only if $A \in \C$, and any object of $\A/\C$ has the form $\QQ(A)$ for some $A \in \A$}.
These are the facts (which essentially characterize the quotient category) which allow us to safely apply the ``$\operatorname{mod} \C$'' language, since this language signifies simply that we are in the quotient abelian category.
\end{quote}
As \nameft{Grothendieck} indicated by the word ``essentially'' all the above conditions yet do not characterize \nameft{Serre} quotients:
In Appendix~\ref{sec:HT} we give an explicit example of an exact faithful functor between \nameft{Abel}ian categories which is not full on its image.

\nameft{Gabriel} proved in his thesis (cf.~Proposition~\ref{prop:local_ker}) the following characterization of \nameft{Serre} quotients under a stronger assumption:
An exact functor $\QQ:\A\to\B$ induces an equivalence between the \nameft{Serre} quotient category $\A/\ker \QQ$ and $\B$ if $\QQ$ admits a right adjoint for which the counit of the adjunction is an isomorphism.
Such a right adjoint is called a section functor.

One goal of this paper is to prove two propositions which, under weaker assumptions, allow us to recognize \nameft{Serre} quotients not necessarily admitting section functors.
In Proposition~\ref{prop:F_equiv} we were able to prove an analogon of the second isomorphism theorem for \nameft{Abel}ian categories relating two \nameft{Serre} quotients modulo thick torsion subcategories.
In Proposition~\ref{prop:equiv} we give, under an additional assumption, the following analogon of the fundamental homomorphism Theorem.
If $\QQ$ is a certain restriction of an ambient functor $\QQ'$ satisfying the assumptions of \nameft{Gabriel}'s characterization then $\QQ$ still induces the desired equivalence of categories $\A/\ker \QQ\simeq \B$.

Our original motivation is to establish a constructive setup for coherent sheaves on several classes of schemes $X$.
This setup begins with describing quotient categories constructively by the $3$-arrow formalism of so-called the \nameft{Gabriel} morphisms \cite{BL_GabrielMorphisms}.
Then, the methods in this paper allow to recognize $\B=\Coh X$ as a \nameft{Serre} quotient of the \nameft{Abel}ian category $\A=\Sfpgrmod$ of \emph{finitely presented} graded modules over some graded coherent ring $\grS$.
This gives an alternative and simpler proof than the one which can be derived from \cite[Theorem~2.6]{KrauseSpectrum} using a result in \cite{Lenzigfpmod}.
The last step in this setup is modeling the category $\A=\Sfpgrmod$ over a computable ring $\grS$ constructively through a presentation matrix \cite{BL}.
This setup turns out to be suitable and computationally efficient for a computer implementation.

In Section~\ref{section:preliminaries_sere_quotients} we collect some preliminaries about \nameft{Serre} quotients.
In Section~\ref{section:recognition} we prove Propositions \ref{prop:F_equiv} and \ref{prop:equiv} which serve to identify certain categories as \nameft{Serre} quotient.
In Section~\ref{section:Coh} we apply the two propositions to categories of coherent sheaves.

We would like to thank \nameft{Florian Eisele}, \nameft{Markus Perling}, \nameft{Sebastian Posur}, and \nameft{Sebastian Thomas} for helpful discussions.

\section{Preliminaries on \nameft{Serre} quotients}\label{section:preliminaries_sere_quotients}

In this section we recall some results about \nameft{Serre} quotients.
From now on $\A$ is an \nameft{Abel}ian category.
See \cite{Gab_thesis} for proofs and \cite{BL_Monads} for detailed references.

A non-empty full subcategory $\mathcal{C}$ of an \nameft{Abel}ian category $\A$ is called \textbf{thick} if it is closed under passing to subobjects, factor objects, and extensions.
In this case the \textbf{(\nameft{Serre}) quotient category} $\A/\C$ is a category with the same objects as $\A$ and $\Hom$-groups
\[
  \Hom_{\A/\C}(M,N) := \varinjlim_{\substack{M' \hookrightarrow M, N' \hookrightarrow N \\ M/M', N' \in \C}} \Hom_\A(M',N/N')\mbox{.}
\]
The \textbf{canonical functor} $\QQ:\A \to \A/\C$ is defined to be the identity on objects and maps a morphism $\phi \in \Hom_\A(M,N)$ to its image in the direct limit $\Hom_{\A/\C}(M,N)$.
The category $\A/\C$ is \nameft{Abel}ian and the canonical functor $\QQ: \A \to \A/\C$ is exact and fulfills the following universal property.
If $\GG:\A \to \D$ an exact functor of \nameft{Abel}ian categories, and $\GG(\C)$ is zero then there exists a unique functor $\HH:\A/\C \to \D$ with $\GG=\HH \circ \QQ$.

Let $\C\subset\A$ be thick.
An object $M \in \A$ is called \textbf{$\C$-torsion-free} if $M$ has no nonzero subobjects in $\C$.
Denote by $\A_\C \subset \A$ the pre-\nameft{Abel}ian full subcategory of $\C$-torsion-free objects.
If every object $M\in\A$ has a maximal subobject $H_\C(M)\in\C$ then $(\C,\A_\C)$ is a \textbf{hereditary torsion theory} of $\A$, i.e., $\C$ and $\A_\C$ are additive and full subcategories, $\C$ is closed under subobjects, $M/H_\C(M)$ is $\C$-torsion-free, i.e., lies in $\A_\C$,
\begin{align*}
  \Hom_\A(C,A) &=0 \mbox{ for all }C\in\C\mbox{ and }A\in\A_\C\mbox{,} \\
  \Hom_\A(C,A) &=0 \mbox{ for all }C\in\C\mbox{ implies }A\in\A_\C\mbox{ for all }A\in\A\mbox{, and} \\
  \Hom_\A(C,A) &=0 \mbox{ for all }A\in\A_\C\mbox{ implies }C\in\C\mbox{ for all }C\in\A\mbox{.}
\end{align*}
In this case we call $\C$ a \textbf{thick torsion} subcategory.

\begin{rmrk} \label{rmrk:Hom_avoiding_limit}
For $\C\subset\A$ thick torsion the description of $\Hom$-groups in $\A/\C$ simplifies to
\[
  \Hom_{\A/\C}(M,N) = \varinjlim_{\substack{M' \hookrightarrow M \\ M / M' \in \C}} \Hom_\A(M',N/H_\C(N)) \mbox{.}
\]
\end{rmrk}

An object $M\in\A$ is called \textbf{$\C$-saturated} if it is $\C$-torsion-free and every extension of $M$ by an object $C \in \C$ is trivial.
Denote by $\Sat_\C(\A)\subset\A$ the full subcategory of $\C$-saturated objects.
We say that $\A$ \textbf{has enough $\C$-saturated objects} if for each $M \in \A$ there exists a $\C$-saturated object $N$ and a morphism $\eta_M: M \to N$ such that $\ker \eta_M \in \C$.
Any thick subcategory $\C\subset\A$ is called a \textbf{localizing} subcategory if the canonical functor $\QQ: \A \to \A/\C$ admits a right adjoint $\SS:\A/\C \to \A$, called the \textbf{section functor} of $\QQ$.
The category $\C \subset \A$ is loca\-lizing if and only if $\A$ has enough $\C$-saturated objects.
The section functor $\SS:\A/\C \to \A$ is left exact and preserves products, the counit of the adjunction $\delta: \QQ \circ \SS \xrightarrow{\sim} \Id_{\A/\C}$ is a natural isomorphism, and an object $M$ in $\A$ is $\C$-saturated if and only if $\eta_M:M \to (\SS \circ \QQ)(M)$ is an isomorphism, where $\eta$ is the unit of the adjunction.
Finally, a localizing $\C\subset\A$ is a thick torsion subcategory with $H_\C(M) = \ker \eta_M$.

\section{Recognition of \nameft{Serre} quotients}\label{section:recognition}

We say that an exact functor  $\QQ' :\biggerA \to \B'$  of \nameft{Abel}ian categories admits a \textbf{section functor} if $\QQ'$ admits a right adjoint $\SS'$ such that the counit of the adjunction $\delta':\QQ' \circ \SS' \to \Id_{\B'}$ is a natural isomorphism\footnote{These adjunctions $\QQ' \dashv (\SS':\B' \to \biggerA)$ are the exact reflective localization of \nameft{Abel}ian categories.}.
It follows that $\QQ'$ is essentially surjective.
The next proposition characterizes \nameft{Serre} quotients $\biggerA/\biggerC$ for $\biggerC$ a \emph{localizing} subcategory of $\biggerA$.

\begin{prop}[{\cite[Proposition~III.2.5]{Gab_thesis}, \cite[Chap.~1.2.5.d]{GabZis}}]\label{prop:local_ker}
  Let $\QQ':\biggerA \to \B'$ be an exact functor of \nameft{Abel}ian categories admitting a section functor $\SS'$.
  Then $\biggerC := \ker \QQ'$ is a localizing subcategory of $\biggerA$ and the adjunction $\QQ' \dashv (\SS':\B' \to \biggerA)$ induces an adjoint equivalence of categories $\biggerA/\biggerC \simeq \B'$.
\end{prop}

The aim of this section is to formulate a characterization of certain \nameft{Serre} quotients $\A \to \A/\C$ where the thick subcategory $\C$ is \emph{not} necessarily localizing.
The following two propositions are analogous to the second isomorphism Theorem and to the fundamental homomorphism Theorem.
However, they need some additional assumptions.

\begin{wrapfigure}[8]{r}{2cm}
\centering
\vskip -1.1cm
\begin{tikzpicture}
  \node (A) {$\bigA$};
  \node (C) at ($(0,-1.5)+(A)$) {$\bigC$};
  \node (tA) at ($(1,-1)+(A)$) {$\ourA$};
  \node (tC) at ($(C)+(tA)$) {$\ourC$};
  
  \draw (A) -- (C);
  \draw (A) -- (tA);
  \draw (tA) -- (tC);
  \draw (C) -- (tC);
  
\end{tikzpicture}
\end{wrapfigure}
\mbox{}
\vspace{-1em}
\begin{prop}[Second isomorphism theorem] \label{prop:F_equiv}
  Let $\bigA$ be an \nameft{Abel}ian category and $\bigC \subset \bigA$ a thick torsion subcategory.
  Then for each thick subcategory $\ourA \subset \bigA$ the intersection $\ourC := \bigC \cap \ourA$ is a thick torsion subcategory of $\ourA$.
  If furthermore the restricted canonical functor $\ourA\to\bigA/\bigC$ is essentially surjective then it induces and equivalence of categories
  \[
    \ourA/\ourC \xrightarrow{\sim} \bigA/\bigC \mbox{.}
  \]
\end{prop}
\begin{proof}
  From the thickness of $\ourA \subset \bigA$ it follows that for each object $N \in \ourA$ the maximal $\bigC$-subobject $H_\bigC(N)$ lies in $\ourC$.
  Furthermore $N/H_\bigC(N)$ is $\ourC$-free since it is $\bigC$-free.
  Summing up, $H_\ourC(N):=H_\bigC(N)$ is the maximal $\ourC$-subobject of $N \in \ourA$ establishing the first assertion.
  By the universal property of $\QQ$ there exists a functor $\ourA/\ourC \to \bigA/\bigC$ which is essentially surjective by our assumption.
  We will now show that it is fully faithful.
  Let $M,N\in\ourA$.
  We can replace $N$ by its $\ourC$-free factor $N / H_\ourC(N)$ and without loss of generality assume that $N$ is $\ourC$-free.
  Because of the thickness of $\ourA\subset\bigA$ and the definition of $\ourC$, the $\bigA$-subobjects $M'$ of $M$ with $M/M'\in \bigC$ are exactly the $\ourA$-subobjects with $M/M'\in \ourC$ and we obtain
  \begin{align*}
    \Hom_{\ourA/\ourC}(M,N) 
      &= \lim_{\substack{ M' \hookrightarrow M\text{ in }\ourA, \\ M/M' \in \ourC}} \Hom_{\ourA}( M', N )
      &\mbox{by Remark~\ref{rmrk:Hom_avoiding_limit} and $N\in \ourA_\ourC$}\\
      &= \lim_{\substack{ M' \hookrightarrow M\text{ in }\bigA, \\ M/M' \in \bigC}} \Hom_{\ourA}( M', N )\\
      &\cong \lim_{\substack{ M' \hookrightarrow M\text{ in }\bigA, \\ M/M' \in \bigC}} \Hom_{\bigA}( M', N )
      & \ourA\subset\bigA\mbox{ full}\\
      &= \Hom_{\bigA/\bigC}(M,N)
      &\mbox{by Remark~\ref{rmrk:Hom_avoiding_limit} and $N\in \bigA_\bigC$.}
  \end{align*}  
\end{proof}

We say that an exact and essentially surjective functor $\QQ:\ourA \to \B$ of \nameft{Abel}ian categories \textbf{admits a sections functor up to extension} if there exists an exact functor $\QQ': \biggerA \to \B'$ admitting a section functor with $\B \subset \B'$ a replete and full \nameft{Abel}ian subcategory, $\ourA \subset \biggerA$ thick, and $\QQ={\QQ'}_{|\ourA}: \ourA \to \B$.
Now we can formulate the analogon of the fundamental homomorphism Theorem.

\begin{wrapfigure}[10]{r}{3.5cm}
\centering
\vskip -0.5cm
\begin{tikzpicture}
  \node (tA) {$\bigA$};
  \node (A1) at ($(0,1.5)+(tA)$) {$\biggerA$};
  \node (tC) at ($(0,-1.5)+(tA)$) {$\bigC$};
  \node (A) at ($(1,-1)+(A)$) {$\ourA$};
  \node (C) at ($(tC)+(A)$) {$\ourC$};
  \node (B) at ($(tA)+(-1.5,0)$) {$\B$};
  \node (B1) at ($(A1)+(B)$) {$\B'$};
  \node (B0) at ($(tC)+(B)$) {$0$};
  
  \draw (A) -- (C);
  \draw (A) -- (tA) -- (A1);
  \draw (tA) -- (tC);
  \draw (C) -- (tC);
  \draw (B0) -- (B) -- (B1);
  
  \draw[-stealth'] (A1) -- node[above] {$\QQ'$} (B1);
  \draw[-stealth'] (tA) -- node[above] {$\widetilde{\QQ}$} (B);
  \draw[-stealth'] (tC) --  (B0);
  
\end{tikzpicture}
\end{wrapfigure}
\mbox{}
\vspace{-1em}
\begin{prop}[Fundamental homomorphism theorem] \label{prop:equiv}
  Let $\QQ:\ourA \to \B$ be an exact and essentially surjective functor of \nameft{Abel}ian categories which admits a section functor up to extension.
  Then $\QQ$ induces an equivalence of categories $\ourA/\ourC \simeq \B$, where $\ourC :=\ker \QQ$ is a thick torsion subcategory of $\ourA$.
\end{prop}
\begin{proof}
  By assumption there exists an exact functor $\QQ': \biggerA \to \B'$ of \nameft{Abel}ian categories admitting a section functor $\SS'$ with $\B \subset \B'$ a replete and full \nameft{Abel}ian subcategory, $\ourA \subset \biggerA$ thick, and $\QQ={\QQ'}_{|\ourA}: \ourA \to \B$.
  First note that $\ourC = \ker \QQ= \bigC \cap \ourA$ for $\bigC := \ker \QQ'$.
  Define $\bigA \subset \biggerA$ as the preimage\footnote{Note that we didn't need the preimage $\bigA$ in the statement of the proposition.} of $\B$ under $\QQ'$, i.e., the full subcategory of $\biggerA$ with object class $\Obj \bigA = \{ M' \in \biggerA \mid \QQ'(M') \in \B \}$.
  $\bigA$ is a full replete \nameft{Abel}ian subcategory of $\biggerA$.
  The section functor $\SS'$ maps objects in $\B$ to objects in $\bigA$ since the counit $\delta': \QQ' \circ \SS' \to \Id_{\B'}$ is an isomorphism.
  Hence, the adjunction induces a restricted adjunction $\widetilde{\QQ} \dashv (\widetilde{\SS}: \B \to \bigA)$ and the counit $\widetilde{\delta}:\widetilde{\QQ} \circ \widetilde{\SS} \to \Id_{\B}$ is still an isomorphism.
  Proposition~\ref{prop:local_ker} implies that $\bigA/\bigC \simeq \B$ since $\bigC=\ker \widetilde{\QQ}=\ker \QQ' \subset \bigA$.
  
  The assertion $\B \simeq \bigA/\bigC \simeq \ourA / \ourC$ now follows from Proposition~\ref{prop:F_equiv} once we have shown that $\ourA\to\bigA/\bigC$ is essentially surjective.
  As $\Sat_\bigC(\bigA)\to\bigA/\bigC$ is essentially surjective (even an equivalence) we need to show that for every $\overline{M}\in\Sat_\bigC(\bigA)$ there exists an $M\in\ourA$ and $M\to\overline{M}$ with kernel and cokernel in $\bigC$.
  Let $M \in \ourA$ be a preimage of $\widetilde{\QQ}(\overline{M}) \in \B$ under the essentially surjective restriction $\QQ=\QQ'_{|\ourA}=\widetilde{\QQ}_{|\ourA}: \ourA \to \B$.
  Then
  \[
   (\widetilde{\SS} \circ \widetilde{\QQ})(M) \xrightarrow[\widetilde{\QQ}(M) \cong \widetilde{\QQ}(\overline{M})]{\cong} (\widetilde{\SS} \circ \widetilde{\QQ})(\overline{M}) \xrightarrow[\cong]{\widetilde{\eta}_{\overline{M}}^{-1}} \overline{M} \mbox{.}
  \]
  Furthermore,  $M \xrightarrow{\widetilde{\eta}_{M}} (\widetilde{\SS} \circ \widetilde{\QQ})(M)$ has kernel and cokernel in $\bigC$.
\end{proof}

\section{Applications to coherent sheaves} \label{section:Coh}

\subsection{Coherent sheaves on projective schemes} \label{subsec:Proj}

Let $A$ be a commutative unitial ring and $\grS=A[x_0,\ldots,x_n]$ the $\Z$-graded polynomial ring over $A$ with $\deg x_i=1$ for all $i$.
Let $\PP^n_A=\Proj \grS$ the $n$-dimensional projective space $A$.
Denote by $\Coh \PP^n_A$ the category of coherent sheaves over $\PP^n_A$.

The category $\Sqfgrmod$ of \textbf{quasi finitely generated} graded $\grS$-modules is the full subcategory in the category of (not necessarily finitely generated) graded $\grS$-modules $\grM$ where the truncated submodule $\grM_{\ge d}$ is finitely generated for $d\in\Z$ high enough.
Further, denote by $\Sqfgrmod^0$ its thick subcategory of $\grS$-modules $\grM$ with $\grM_{\geq d}=0$ for $d \in \Z$ high enough.

\begin{theorem}[\cite{FAC,EGA2}\footnote{The case of $A$ a field is treated in \cite[Prop.~III.2.3, Prop.~III.2.5, Theo.~III.2.2, Prop.~III.2.7, Prop.~III.2.8, Prop.~III.3.6]{FAC} and the general case in \cite[Prop.~3.2.4, Prop.~3.4.3.ii, Prop.~3.3.5, Theorem.~3.4.4]{EGA2}.
The adjointness is shown there by proving the zig-zag identities.}]
  The sheafification functor $\Sh: \Sqfgrmod \to \Coh \PP^n_A$ is exact with kernel $\Sqfgrmod^0$.
  The functor $\Gamma_\bullet: \Coh \PP^n_A \to \Sqfgrmod$ is right adjoint to $\Sh$ and the counit $\widetilde{\delta}:\Sh\circ\Gamma_\bullet\to\Id_{\Coh \PP^n_A}$ is an isomorphism.
  The adjunction $\Sh \dashv \Gamma_\bullet$ induces by Proposition~\ref{prop:local_ker} an adjoint equivalence of categories
  \begin{center}
    \begin{tikzpicture}
      \coordinate (offset) at (0,0.15);
      \node (A) {$\Coh \PP_A$};
      \node (B) at ($(A)+(4,0)$) {$\displaystyle\frac{\Sqfgrmod}{\Sqfgrmod^0}$.};
      \draw[-stealth'] ($(A.east)+(offset)$) -- node[above] {$\Gamma_\bullet$} ($(B.west)+(offset)$);
      \draw[stealth'-] ($(A.east)-(offset)$) -- node[below] {$\Sh$} node[above=-3pt] {\tiny$\sim$} ($(B.west)-(offset)$);
    \end{tikzpicture}
  \end{center}
\end{theorem}

Let $\Sfpgrmod$ denote the category of finitely presented graded $\grS$-modules and $\Sfpgrmod^0$ its thick subcategory of $\grS$-modules $\grM$ with $\grM_{\ge d}=0$ for all $d$ large enough.

\begin{coro}
  Let $A$ be a \nameft{Noether}ian ring\footnote{This guarantees that $\Sfpgrmod$ is an \nameft{Abel}an subcategory of $\Sqfgrmod$.}.
  The exact and essentially surjective sheafification functor $\Sh: \Sfpgrmod \to \Coh \PP^n_A$ induces an equivalence of categories
  \[
     \frac{\Sfpgrmod}{\Sfpgrmod^0} \simeq \Coh \PP^n_A\mbox{,}
  \]
  where $\Sfpgrmod^0$ coincides with the kernel of the sheafification functor.
\end{coro}
\begin{proof}
  Define $\A := \Sfpgrmod \subset \Sqfgrmod =: \widetilde{\A} \supset \bigC := \Sqfgrmod^0$ and $\ourC := \Sfpgrmod^0 = \ourA \cap \bigC$.
  To apply Proposition~\ref{prop:F_equiv} to the thick subcategory $\ourA \subset \bigA$ we need to show that the restricted canonical functor $\widetilde{\QQ}_{\mid \ourA}:\ourA \to \bigA/\bigC$ is essentially surjective.
  By definition of $\Sqfgrmod$ for each $\grM \in \Sqfgrmod$ there exists a $d=d(\grM)\in \Z$ high enough such that $\grM_{\geq d}$ is finitely generated, i.e., lies in $\Sfpgrmod$.
  Further, $\grM/\grM_{\geq d}$ lies in $\Sqfgrmod^0$ by definition of the latter.
  Hence $\grM_{\geq d}$ and $\grM$ have isomorphic images in $\frac{\Sqfgrmod}{\Sqfgrmod^0}$ and we obtain
  \[
    \frac{\Sqfgrmod}{\Sqfgrmod^0} \simeq \frac{\Sfpgrmod}{\Sfpgrmod^0} \simeq \Coh \PP^n_A \mbox{.} \qedhere
  \]
\end{proof}

\begin{rmrk} \label{rmrk:no_section_functor}
  Although the two \nameft{Serre} quotient categories $\frac{\Sqfgrmod}{\Sqfgrmod^0} \simeq \frac{\Sfpgrmod}{\Sfpgrmod^0}$ yield equi\-valent representations of the category $\Coh \PP^n_A$ we now show that the thick torsion subcategory $\Sfpgrmod^0 \subset \Sfpgrmod$ is not localizing, even though $\Sqfgrmod^0 \subset \Sqfgrmod$ is.
  
  Therefore, let $\grS=k[x,y]$ for a field $k$ and assume that there exists a functor $\SS: \frac{\Sfpgrmod}{\Sfpgrmod^0}\to\Sfpgrmod$ right adjoint to the canonical functor $\QQ:\Sfpgrmod\to\frac{\Sfpgrmod}{\Sfpgrmod^0}$.
  Define $M:=\grS/\langle y\rangle$ and $N^n:=x^{-n}M$ for each $n\in\Z$.
  The sheafification of $M$ is a skyscraper sheaf on $\PP^1_k$.
  As $\QQ(N^n)\cong \QQ(M)$ there exists by the $\Hom$-adjunction a morphism $\varphi^n: N^n\to (\SS \circ \QQ)(M)$ with kernel and cokernel in $\Sfpgrmod^0$.
  As $N^n$ and $M$ are not in $\Sfpgrmod^0$ the morphism $\varphi^n$ is nonzero.
  Thus, $\varphi^n$ must map the cyclic generator of $N^n$ of degree $-n$ to a nonzero element of degree $-n$ in $(\SS \circ \QQ)(M)$.
  In particular, $\dim_k ((\SS \circ \QQ)(M))_{-n}>0$ for all $n \in \Z$ and $(\SS \circ \QQ)(M)$ is not finitely generated.
  This is a contradiction.
\end{rmrk}

In \cite{BL_Sheaves} we will prove that $\Coh \PP^n_A$ admits for each $d \in \Z$ yet another representation as the \nameft{Serre} quotient $\Sfpgrmodd/\Sfpgrmod_{\geq d}^0$, where $\Sfpgrmodd$ is the category of finitely presented graded $\grS$-modules vanishing in degrees $<d$.
In this case $\Sfpgrmod_{\geq d}^0 = \Sfpgrmod \cap \Sqfgrmod_{\geq d}^0 \subset \Sfpgrmodd$ is localizing.

Although these categories of truncated modules are computable they have the following computational disadvantage.
Given a fixed $d\in \Z$ there exists coherent sheaves $\F \in \Coh \PP^n_A$ for which $\Sfpgrmod$ contains a vastly more efficient model for $\F$ than $\Sfpgrmodd$.
For example, the minimal number of generators of the smallest model for $\O_{\PP^n_A}(k)$ in $\Sfpgrmodd$ is $\max\left(1,\binom{n+k+d}{n}\right)$ which is disadvantageous if $k \gg -d$.
These positively twisted line bundles occur as soon as we need to dualize a locally free resolution.

\subsection{Coherent sheaves on toric varieties}

We refer the reader to \cite{CLS11} for notation.
Let $X_\Sigma$ be a toric variety with no torus factors and \nameft{Cox} ring $\grS=\CC[x_\rho \mid \rho \in \Sigma(1)]$ graded by the divisor class group $\Cl X_\Sigma$.
We denote by $\Sgrmod$ the category of graded $\grS$-modules.

By \cite[Theorem~1.1]{Mus02} (cf.\ also \cite[Prop.~6.A.3]{CLS11}) the global section functor
\[
  \Gamma_\bullet: \qCoh X_\Sigma \to \Sgrmod: \shF \mapsto \bigoplus_{\alpha \in \Cl X_\Sigma} \Gamma(X_\Sigma,\shF(\alpha))
\]
is right inverse but not right adjoint to the exact sheafification functor $\Sh: \Sgrmod \to \qCoh X_\Sigma$.
Recently \nameft{Perling} found the right adjoint of $\Sh$, also valid in the singular case.
\begin{theorem}[{\cite[Theorem~3.8 and Remark~3.9]{PerLift}}] \label{thm:adjointness}
  Let $X_\Sigma$ be a toric variety with no torus factor.
  There exists a so-called \textbf{lifting functor} $\widehat{\Gamma}$ right adjoint to the exact sheafification functor $\Sh: \Sgrmod \to \qCoh X_\Sigma$, where the counit of the adjunction $\delta: \Sh \circ \widehat{\Gamma} \to \Id_{\qCoh X_\Sigma}$ is a natural isomorphism.
\end{theorem}

This allows us to characterize toric sheaves as quotients of finitely generated modules.

\begin{coro}
  Let $X_\Sigma$ be a toric variety with no torus factor.
  The exact and essentially surjective sheafification functor $\Sh: \Sfpgrmod \to \Coh X_\Sigma$ induces the equivalence
  \[
    \frac{\Sfpgrmod}{\Sfpgrmod^0} \simeq \Coh X_\Sigma \mbox{,}
  \]
   of categories where $\Sfpgrmod^0$ is defined as the kernel of the sheafification functor.
\end{coro}
\begin{proof}
  The essential surjectivity of $\Sh$ is the statement of \cite[Cor.~1.2]{Mus02} (cf.\ also \cite[Prop.~6.A.4]{CLS11}).
  Thus, the assumption of Proposition~\ref{prop:equiv} is fulfilled with $\QQ' \dashv (\SS': \B' \to \biggerA)$ being the adjunction $\Sh \dashv (\widehat{\Gamma}:\qCoh X_\Sigma \to \Sgrmod)$ from Theorem~\ref{thm:adjointness}.
\end{proof}

In fact \nameft{Perling} proved Theorem~\ref{thm:adjointness} in a more general setup described in \cite{PerLift} following \cite{Hau08,CoxRings}.
This setup covers toric varieties over arbitrary fields, \nameft{Mori} dream spaces, and categories of equivariant coherent sheaves on them.
Furthermore, \nameft{Trautmann} and \nameft{Perling} proved in \cite[Proposition~5.6.(2),(4)]{PT10} that the sheafification functor restricted to the subcategory of finitely generated graded modules is essentially surjective onto the category of coherent sheaves.

\begin{appendix}

\section{The nonexistence of a fundamental homomorphism theorem} \label{sec:HT}

In this appendix we show that a naive fundamental homomorphism theorem of exact functors between \nameft{Abel}ian categories cannot exist.
Such a theorem would imply that the corestriction of an exact faithful functor to its image\footnote{The image of a functor $\GG:\A \to \B$, denoted by $\img \GG$, is the smallest subcategory of $\B$ which contains all image morphisms $\GG(\phi)$.} is full.
To justify our counterexample we need the following argument.

\begin{lemm} \label{lemm:conservative}
  An exact faithful functor of \nameft{Abel}ian categories is conservative\footnote{A functor $\FF$ is called conservative if reflects isomorphisms, i.e., if $\FF(\phi)$ iso $\implies$ $\phi$ iso.}.
\end{lemm}
\begin{proof}
  Let $\FF: \A \to \B$ be such a functor and $\phi:M \to N$ be a morphism in $\A$.
  Consider the exact $\A$-sequence $0 \to \ker \phi \xrightarrow{\iota} M \xrightarrow{\phi} N \xrightarrow{\pi} \coker \phi \to 0$.
  As $\B$ is \nameft{Abel}ian and $\FF$ exact, $\FF(\phi)$ is an isomorphism iff the morphisms $\FF(\iota)$ and $\FF(\pi)$ are zero in $\B$.
  The faithfulness of $\FF$ implies that then $\iota$ and $\pi$ are zero.
  Finally, since $\A$ is \nameft{Abel}ian it follows that $\phi$ is an isomorphism.
\end{proof}

\begin{exmp}
  Let $\B$ be the category of $k$-vector spaces, $G$ a nontrivial group, and $\A$ the category of $G$-representations with object $(V,\rho:G \to \operatorname{GL}(V))$ and $G$-equivariant morphisms in $\B$.
  The forgetful functor $\FF:\A \to \B, (V,\rho) \mapsto V$ is exact and faithful with $\img \FF = \B$.
  First note that $1_V \in \Hom_\B(V,V)$ is in the image of $1_{(V,\rho)}$ under $\FF$.
  However, $1_V \in \Hom_\B(V,V)$ cannot be in the image under $\FF$ of $\Hom_\A((V,\rho),(V,\rho'))$ for two inequivalent representations $(V,\rho) \not\cong (V,\rho')$; a preimage would be an $\A$-isomorphism since $\FF$ is conservative by Lemma~\ref{lemm:conservative}.
  A contradiction.
\end{exmp}

This shows that the corestriction of an exact faithful functor to its image is not necessarily full, in particular, it is not an equivalence\footnote{
Enlarging the image only deepens the problem.
In the following example, due to \nameft{Eisele} \cite{EiseleWideNotFull},
an inclusion functor of an \nameft{Abel}ian subcategory has the inequivalent target category as its essential image:

  Let $k$ be a field and $(R,\mathfrak{m})$ be a finite dimensional local $k$-algebra with $R/\mathfrak{m} \cong k$.
  Let $\ourA$ be the category of $R$-modules and $\B$ be the category with the same objects as $\A$ but $\Hom_\B(M,N) := \Hom_k(M,N)$, i.e., all $k$-vector space homomorphisms.
  The forgetful (identity on objects) functor $\FF:\ourA\to\B$ is clearly exact and surjective (on objects).
  $\B$ is equivalent to the category of all $k$-vector spaces as every $k$-vector space can be seen as an $R$-module via $k\cong R/\mathfrak{m}$.
  The kernel of $\FF$ is the subcategory $0_\ourA$ of zero objects in $\ourA$.
  But $\ourA/0_\ourA\simeq \ourA$ is not equivalent to $\B$ (if $R\not=k$).
} of categories.

\smallskip
Any exact functor $\GG:\A \to \B$ of \nameft{Abel}ian categories with kernel $\C := \ker \GG$ induces a unique faithful exact functor $\HH:\A/\C \to \B$ such that $\GG = \HH \circ \QQ$.
This is the universal property of the canonical functor $\QQ:\A \to \A/\C$.
Note that $\img \QQ$ is in general strictly contained in $\A/\C$ and that $\img \GG$ is strictly contained in $\img \HH$.
In the rest of the appendix we want to define an image-notion for which equality holds in both cases.

We define the \textbf{conservative image} of a functor $\GG:\A \to \B$, denoted by $\conimg \GG$, to be the smallest subcategory of $\B$ which contains $\img \GG$ and inverses of $\B$-isomorphisms in $\img \GG$.
We call $\GG$ \textbf{conservatively surjective} if $\conimg\GG=\B$.
The name is motivated by $\conimg \GG = \img \GG$ for any conservative functor $\GG$.

\begin{lemm}\label{lemm:conservative_surjective}
  Let $\QQ:\A\to\D$, $\GG:\A\to\B$, and $\HH:\D\to\B$ be functors with $\HH\circ\QQ=\GG$ and $\QQ$ conservatively surjective.
  Then $\conimg \GG=\conimg \HH$.
\end{lemm}
\begin{proof}
  It is clear that $\conimg \HH \supset \conimg \GG$.
  Every morphism $\alpha$ in $\conimg \HH$ is a finite composition of morphisms of the form $\HH(\phi_i)^{\epsilon_i}$, with $\epsilon_i \in \{\pm 1\}$.
  Since $\conimg \QQ=\D$ each $\phi_i$ is in turn a finite composition of morphisms of the form $\QQ(\psi_{ij})^{\sigma_{ij}}$, with $\sigma_{ij} \in \{\pm 1\}$.
  Finally $\alpha$ is then a finite composition of morphisms of the form $\GG(\psi_{ij})^{\epsilon_i \sigma_{ij}}$.
\end{proof}

\begin{prop}
  Let  $\GG: \A \to \B$ be an exact functor of \nameft{Abel}ian categories, $\C := \ker \GG$, $\HH:\A/\C \to \B$ the unique functor such that $\GG=\HH \circ \QQ$.
  Then $\HH$ is an exact, faithful, and conservative functor.
  Furthermore $\img \HH = \conimg \HH = \conimg \GG$.
\end{prop}
\begin{proof}
  $\HH$ is exact and faithful by definition and hence conservative by Lemma~\ref{lemm:conservative}.
  Hence $\img \HH = \conimg \HH$.
  The canonical functor $\QQ$ is conservatively surjective, as every morphism $\overline{\phi}:M \to N$ in $\A/\C$ is of the form $\QQ(M' \stackrel{\iota}{\hookrightarrow} M)^{-1} \QQ(M' \xrightarrow{\phi} N/N') \QQ(N \stackrel{\pi}{\twoheadrightarrow} N/N')^{-1}$, with $M/M',N' \in \C$.
  The equality $\conimg \HH = \conimg \GG$ follows from Lemma~\ref{lemm:conservative_surjective}.
\end{proof}

\end{appendix}

\def\cprime{$'$} \def\cprime{$'$} \def\cprime{$'$} \def\cprime{$'$}
  \def\cprime{$'$}
\providecommand{\bysame}{\leavevmode\hbox to3em{\hrulefill}\thinspace}
\providecommand{\MR}{\relax\ifhmode\unskip\space\fi MR }
\providecommand{\MRhref}[2]{%
  \href{http://www.ams.org/mathscinet-getitem?mr=#1}{#2}
}
\providecommand{\href}[2]{#2}

\end{document}